\documentclass[12pt]{article}
\usepackage{amsmath,amsthm, amssymb}
\usepackage{amssymb,latexsym}

\newtheorem{theorem}{Theorem}

\newtheorem{lemma}{Lemma}

\begin{document}

\baselineskip=17pt

\title{\bf Diophantine approximation by  Piatetski-Shapiro primes}

\author{\bf S. I. Dimitrov}
\date{2020}
\maketitle
\begin{abstract}
Let $[\, \cdot\,]$ be the floor function.
In this paper we show that whenever $\eta$ is real,
the constants $\lambda _i$ satisfy some necessary conditions,
then for any fixed $1<c<38/37$ there exist  infinitely many prime
triples $p_1,\, p_2,\, p_3$ satisfying the inequality
\begin{equation*}
|\lambda _1p_1 + \lambda _2p_2 + \lambda _3p_3+\eta|<(\max p_j)^{{\frac{37c-38}{26c}}}(\log\max p_j)^{10}
\end{equation*}
and such that $p_i=[n_i^c]$, $i=1,\,2,\,3$.
\quad\\
{\bf  2000 Math.\ Subject Classification}: 11D75  $\cdot$  11P32
\par
\end{abstract}
\section{Introduction and statement of the result}
\indent

Let $\mathbb{P}$ denotes the set of all prime numbers.
In 1953 Piatetski-Shapiro \cite{Shapiro} showed that for any fixed $\gamma\in(11/12,1)$ the set
\begin{equation*}
\mathbb{P}_\gamma=\{p\in\mathbb{P}\;\;|\;\; p= [n^{1/\gamma}]\;\; \mbox{ for some } n\in \mathbb{N}\}
\end{equation*}
is infinite.
The prime numbers of the form $p = [n^{1/\gamma}]$ are called Piatetski-Shapiro primes of type $\gamma$.
Subsequently  the interval for $\gamma$ was sharpened many times and the best result
up to now belongs to Rivat and Wu \cite{Rivat-Wu} for $\gamma\in(205/243,1)$.

Twenty years later Vaughan \cite{Vaughan} proved that whenever $\delta>0$,
$\eta$ is real and constants $\lambda_i$ satisfy some conditions,
there are infinitely many prime triples $p_1,\,p_2,\,p_3$ such that
\begin{equation}\label{Inequality1}
|\lambda_1p_1+\lambda_2p_2+\lambda_3p_3+\eta|<(\max p_j)^{-\xi+\delta}
\end{equation}
for $\xi=1/10$.
Latter the upper bound for $\xi $ was improved several times
and the best result up to now is due to K. Matom\"{a}ki \cite{Matomaki} with $\xi=2/9$.

In order to establish our result we solve the inequality \eqref{Inequality1}
with Piatetski-Shapiro primes. Thus we prove the following theorem.

\begin{theorem}\label{Theorem}
Suppose that $\lambda_1,\lambda_2,\lambda_3$ are non-zero real numbers, not all of
the same sign, $\eta$ is real, $\lambda_1/\lambda_2$ is irrational
and $\gamma$ be fixed with $37/38<\gamma<1$.
Then there exist infinitely many ordered triples of Piatetski-Shapiro primes $p_1,\,p_2,\,p_3$
of type $\gamma$ such that
\begin{equation*}
|\lambda_1p_1+\lambda_2p_2+\lambda_3p_3+\eta|<(\max p_j)^{^{\frac{37-38\gamma}{26}}}(\log\max p_j)^{10}\,.
\end{equation*}
\end{theorem}

\section{Notations}
\indent

The letter $p$ will always denote prime number.
As usual $[t]$ and $\{t\}$ we denote the integer part of $x$ and the fractional part of $x$.
Moreover $\psi(t)=\{t\}-1/2$. We write  $e(t)$=exp($2\pi it$).
Let $\gamma$ be a real constant such that $37/38<\gamma<1$.
Since $\lambda_1/\lambda_2$ is irrational, there are infinitely many different convergents
$a_0/q_0$ to its continued fraction, with
\begin{equation}\label{lambda12a0q0}
\bigg|\frac{\lambda_1}{\lambda_2} - \frac{a_0}{q_0}\bigg|<\frac{1}{q_0^2}\,,\quad (a_0, q_0) = 1\,,\quad a_0\neq0
\end{equation}
and $q_0$ is arbitrary large.
Denote
\begin{align}
\label{X}
&X=q_0^{13/6}\,;\\
\label{Delta}
&\Delta=X^{-\frac{12}{13}}\log X\,;\\
\label{varepsilon}
&\varepsilon=X^{^{\frac{37-38\gamma}{26}}}\log^{10}X\,;\\
\label{H}
&H=\frac{\log^2X}{\varepsilon}\,;\\
\label{SalphaX}
&S(\alpha,X)=\sum\limits_{\lambda_0X<p\leq X\atop{p\in\mathbb{P}_\gamma}}
p^{1-\gamma}e(\alpha p)\log p\,,\quad 0<\lambda_{0}<1\,;\\
\label{SigmaalphaX}
&\Sigma(\alpha,X)=\gamma\sum\limits_{\lambda_0X<p\leq X}e(\alpha p)\log p\,;\\
\label{OmegaalphaX}
&\Omega(\alpha,X)=\sum\limits_{\lambda_0X<p\leq X}p^{1-\gamma}\big(\psi(-(p+1)^\gamma)-\psi(-p^\gamma)\big)e(\alpha p)\log p\,;\\
\label{IalphaX}
&I(\alpha,X)=\gamma\int\limits_{\lambda_0X}^{X}e(\alpha y)dy\,.
\end{align}

\newpage
\section{Preliminary lemmas}
\indent

\begin{lemma}\label{Fourier} Let $\varepsilon>0$ and $k\in \mathbb{N}$.
There exists a function $\theta(y)$ which is $k$ times continuously differentiable and
such that
\begin{align*}
&\theta(y)=1\quad\quad\quad\mbox{for }\quad\quad|y|\leq 3\varepsilon/4\,;\\
&0<\theta(y)<1\quad\mbox{for}\quad3\varepsilon/4 <|y|< \varepsilon\,;\\
&\theta(y)=0\quad\quad\quad\mbox{for}\quad\quad|y|\geq \varepsilon\,.
\end{align*}
and its Fourier transform
\begin{equation*}
\Theta(x)=\int\limits_{-\infty}^{\infty}\theta(y)e(-xy)dy
\end{equation*}
satisfies the inequality
\begin{equation*}
|\Theta(x)|\leq\min\bigg(\frac{7\varepsilon}{4},\frac{1}{\pi|x|},\frac{1}{\pi |x|}
\bigg(\frac{k}{2\pi |x|\varepsilon/8}\bigg)^k\bigg)\,.
\end{equation*}
\end{lemma}
\begin{proof}
See (Lemma 1, \cite{Tolev}).
\end{proof}

\begin{lemma}\label{SigmaIasympt} Let $|\alpha|\leq\Delta$.
Then for the sum denoted by \eqref{SigmaalphaX} and the integral denoted by \eqref{IalphaX}
the asymptotic formula
\begin{equation*}
\Sigma(\alpha,X)=I(\alpha,X)+\mathcal{O}\left(\frac{X}{e^{(\log X)^{1/5}}}\right)
\end{equation*}
holds.
\end{lemma}
\begin{proof}
This lemma is very similar to result of Tolev  \cite{Tolev}.
Inspecting the arguments presented in (\cite{Tolev}, Lemma 14),
the reader will easily see that the proof of Lemma \ref{SigmaIasympt}
can be obtained by the same way.
\end{proof}

\begin{lemma}\label{Omegaest}
For the sum denoted by \eqref{OmegaalphaX} the upper bound
\begin{equation*}
\Omega(\alpha,X)\ll X^{\frac{37-12\gamma}{26}}\log^5X\,.
\end{equation*}
holds.
\end{lemma}
\begin{proof}
It follows by the same argument used  in (\cite{Dimitrov}, (36)).
\end{proof}

\begin{lemma}\label{Expsumest}
Suppose that $\alpha \in \mathbb{R}$,\, $a \in \mathbb{Z}$,\, $q\in \mathbb{N}$,\,
$\big|\alpha-\frac{a}{q}\big|\leq\frac{1}{q^2}$\,, $(a, q)=1$.

Let
\begin{equation*}
\Psi(X)=\sum\limits_{p\le X}e(\alpha p)\log p\,.
\end{equation*}
Then
\begin{equation*}
\Psi(X)\ll \Big(Xq^{-1/2}+X^{4/5}+X^{1/2}q^{1/2}\Big)\log^4X\,.
\end{equation*}
\end{lemma}
\begin{proof}
See (\cite{Iwaniec-Kowalski}, Theorem 13.6).
\end{proof}

\section{Outline of the proof}
\indent

Consider the sum
\begin{equation}\label{Gamma}
\Gamma(X)=\sum\limits_{\lambda_0X<p_1,p_2,p_3\leq X\atop{p_i\in\mathbb{P}_\gamma,\, i=1,2,3}}
\theta(\lambda_1p_1+\lambda_2p_2+\lambda_3p_3+\eta)\log p_1\log p_2\log p_3\,.
\end{equation}
Using the inverse Fourier transform for the function $\theta(x)$ we get
\begin{align*}
\Gamma(X)&=\sum\limits_{\lambda_0X<p_1,p_2,p_3\leq X\atop{p_i\in\mathbb{P}_\gamma,\, i=1,2,3}}\log p_1\log p_2\log p_3
\int\limits_{-\infty}^{\infty}\Theta(t)e\big((\lambda_1p_1+\lambda_2p_2+\lambda_3p_3+\eta)t\big)\,dt\\
&=\int\limits_{-\infty}^{\infty}\Theta(t)S(\lambda_1t,X)S(\lambda_2t,X)S(\lambda_3t,X)e(\eta t)\,dt\,.
\end{align*}
We decompose $\Gamma(X)$ as follows
\begin{equation}\label{Gammadecomp}
\Gamma(X)=\Gamma_1(X)+\Gamma_2(X)+\Gamma_3(X)\,,
\end{equation}
where
\begin{align}
\label{Gamma1}
&\Gamma_1(X)=\int\limits_{|t|<\Delta}\Theta(t)S(\lambda_1t,X)S(\lambda_2t,X)S(\lambda_3t,X)e(\eta t)\,dt\,,\\
\label{Gamma2}
&\Gamma_2(X)=\int\limits_{\Delta\leq|t|\leq H}\Theta(t)S(\lambda_1t,X)S(\lambda_2t,X)S(\lambda_3t,X)e(\eta t)\,dt\,,\\
\label{Gamma3}
&\Gamma_3(X)=\int\limits_{|t|>H}\Theta(t)S(\lambda_1t,X)S(\lambda_2t,X)S(\lambda_3t,X)e(\eta t)\,dt\,.
\end{align}
We shall estimate $\Gamma_1(X),\,\Gamma_2(X)$ and $\Gamma_3(X)$, respectively,
in the sections \ref{SectionGamma1}, \ref{SectionGamma2} and \ref{SectionGamma3}.
In section \ref{Sectionfinal} we shall complete the proof of Theorem \ref{Theorem}.

\section{Lower bound  of $\mathbf{\Gamma_1(X)}$}\label{SectionGamma1}
\indent

In order to find the lower bound of $\Gamma_1(X)$ we need to prove the following two lemmas.

\begin{lemma}\label{SIasymptotic}
For the sum denoted by \eqref{SalphaX} and the integral denoted by \eqref{IalphaX}
the asymptotic formula
\begin{equation}\label{S-Iasymptotic}
S(\alpha,X)=I(\alpha,X)+ \mathcal{O}\left(\frac{X}{e^{(\log X)^{1/5}}}\right)
\end{equation}
holds.
\end{lemma}
\begin{proof}
From \eqref{SalphaX} -- \eqref{OmegaalphaX} we have
\begin{align}\label{SOmegaSigma}
S(\alpha,X)&=\sum\limits_{\lambda_0X<p\leq X}p^{1-\gamma}
\big([-p^\gamma]-[-(p+1)^\gamma]\big)e(\alpha p)\log p\nonumber\\
&=\sum\limits_{\lambda_0X<p\leq X}p^{1-\gamma}\big((p+1)^\gamma-p^\gamma\big)e(\alpha p)\log p\nonumber\\
&+\sum\limits_{\lambda_0X<p\leq X}p^{1-\gamma}
\big(\psi(-(p+1)^\gamma)-\psi(-p^\gamma)\big)e(\alpha p)\log p\nonumber\\
&=\Sigma(\alpha,X)+\Omega(\alpha,X)+\mathcal{O}\big(\log^2X\big)\,.
\end{align}
Bearing in mind  \eqref{SOmegaSigma}, Lemma \ref{SigmaIasympt} and Lemma \ref{Omegaest}
we obtain the asymptotic formula \eqref{S-Iasymptotic}.
\end{proof}

\begin{lemma}\label{intSintI} Let $\lambda\neq0$.
Then for the sum denoted by \eqref{SalphaX} and the integral denoted by \eqref{IalphaX} we have
\begin{align*}
&\emph{(i)}\quad\quad\quad\int\limits_{-\Delta}^\Delta|S(\lambda\alpha,X)|^2\,d\alpha\ll X\log^3X\,,\\
&\emph{(ii)}\quad\quad\quad\int\limits_{-\Delta}^\Delta|I(\lambda\alpha)|^2\,d\alpha\ll X\log X\,,\\
&\emph{(iii)}\quad\quad\quad\int\limits_{0}^1|S(\alpha,X)|^2\,d\alpha\ll X^{2-\gamma}\log^2X\,.
\end{align*}
\end{lemma}
\begin{proof} We only prove (i). The inequalities (ii) and (iii) can be proved likewise.

Using \eqref{Delta}, \eqref{SalphaX} and  Lagrange's mean value theorem we obtain
\begin{align*}
\int\limits_{-\Delta}^\Delta|S(\lambda\alpha,X)|^2\,d\alpha&=
\sum\limits_{\lambda_0X<p_1,p_2\leq X\atop{p_i\in\mathbb{P}_\gamma,\, i=1,2}}
(p_1p_2)^{1-\gamma}\log p_1\log p_2\int\limits_{-\Delta}^\Delta e(\lambda(p_1-p_2)\alpha)d\alpha\\
&\ll X^{2-2\gamma}(\log X)^2\sum\limits_{\lambda_0X<n_1,n_2\leq X\atop{n_i=m_i^{1/\gamma},\, i=1,2}}
\min\bigg(\Delta,\frac{1}{|n_1-n_2|}\bigg)\\
&\ll\Delta X^{2-\gamma}\log^2X
+X^{2-2\gamma}(\log X)^2\sum\limits_{\lambda_0X<n_1,n_2\leq X\atop{n_i=m_i^{1/\gamma},\, i=1,2
\atop{n_1<n_2}}}\frac{1}{n_2-n_1}\\
&\ll\Delta X^{2-\gamma}\log^2X
+X^{2-2\gamma}(\log X)^2\sum\limits_{(\lambda_0X)^\gamma<m_1,m_2\leq X^\gamma\atop{m_1<m_2}}
\frac{1}{m_2^{1/\gamma}-m_1^{1/\gamma}-1}\\
&\ll\Delta X^{2-\gamma}\log^2X
+X^{1-\gamma}(\log X)^2\sum\limits_{(\lambda_0X)^\gamma<m_1,m_2\leq X^\gamma\atop{m_1<m_2}}
\frac{1}{m_2-m_1}\nonumber\\
&\ll\Delta X^{2-\gamma}\log^2X+X\log^3X\,.\\
&\ll X\log^3X\,.
\end{align*}
The lemma is proved.
\end{proof}

Put
\begin{align}\label{Si}
&S_i=S(\lambda_it)\,,\\
\label{notIi}
&I_i=I(\lambda_it)\,.
\end{align}
We use the identity
\begin{equation}\label{Identity}
S_1S_2S_3=I_1I_2I_3+(S_1-I_1)I_2I_3+S_1(S_2-I_2)I_3+S_1S_2(S_3-I_3)\,.
\end{equation}
Replace
\begin{equation}\label{JX}
J(X)=\int\limits_{|t|<\Delta}\Theta(t)I(\lambda_1t,X)I(\lambda_2t,X)I(\lambda_3t,X)e(\eta t)\,dt\,.
\end{equation}
Now from \eqref{Gamma1}, \eqref{Identity}, \eqref{JX},
Lemma \ref{Fourier}, Lemma \ref{SIasymptotic} and Lemma \ref{intSintI} it follows
\begin{align}\label{Gamma1-JX}
\Gamma_1(X)-J(X)
&=\int\limits_{|t|<\Delta}\Theta(t)\Big(S(\lambda_1t,X)-I(\lambda_1t,X)\Big)
I(\lambda_2t,X)I(\lambda_3t,X)e(\eta t)\,dt\nonumber\\
&+\int\limits_{|t|<\Delta}\Theta(t)S(\lambda_1t,X)
\Big(S(\lambda_2t,X)-I(\lambda_2t,X)\Big)I(\lambda_3t,X)e(\eta t)\,dt\nonumber\\
&+\int\limits_{|t|<\Delta}\Theta(t)S(\lambda_1t,X)
S(\lambda_2t,X)\Big(S(\lambda_3t,X)-I(\lambda_3t,X)\Big)e(\eta t)\,dt\nonumber\\
&\ll\varepsilon \frac{X}{e^{(\log X)^{1/5}}}\Bigg(\int\limits_{|t|<\Delta}
|I(\lambda_2t,X)I(\lambda_3t,X)|\,dt\nonumber\\
&+\int\limits_{|t|<\Delta}|S(\lambda_1t,X)I(\lambda_3t,X)|dt
+\int\limits_{|t|<\Delta}|S(\lambda_1t,X)S(\lambda_2t,X)|dt\Bigg)\nonumber\\
&\ll\varepsilon \frac{X}{e^{(\log X)^{1/5}}}\Bigg(\int\limits_{|t|<\Delta}
|I(\lambda_2t,X)|^2dt+\int\limits_{|t|<\Delta}|I(\lambda_3t,X)|^2dt\nonumber\\
&+\int\limits_{|t|<\Delta}|S(\lambda_1t,X)|^2dt+\int\limits_{|t|<\Delta}|S(\lambda_2t,X)|^2dt\Bigg)\nonumber\\
&\ll\varepsilon\frac{X^2}{e^{(\log X)^{1/6}}}\,.
\end{align}
On the other hand for the integral defined by \eqref{JX} we write
\begin{equation}\label{JXest1}
J(X)=B(X)+\Phi\,,
\end{equation}
where
\begin{equation*}
B(X)=\gamma^3\int\limits_{\lambda_0X}^{X}
\int\limits_{\lambda_0X}^{X}\int\limits_{\lambda_0X}^{X}
\theta(\lambda_1y_1+\lambda_2y_2+\lambda_3y_3+\eta)
\,dy_1\,dy_2\,dy_3
\end{equation*}
and
\begin{equation}\label{Phi}
\Phi\ll\int\limits_{\Delta}^{\infty }|\Theta(t)|
|I(\lambda_1t,X)I(\lambda_2t,X)I(\lambda_3t,X)|\,dt\,.
\end{equation}
According to (\cite{DimTod}, Lemma 4) we have
\begin{equation}\label{Best}
B(X)\gg\varepsilon X^2\,.
\end{equation}
By  \eqref{IalphaX} we get
\begin{equation}\label{IalphaXest}
I(\alpha,X)\ll\frac{1}{|\alpha|}\,.
\end{equation}
Using \eqref{Phi}, \eqref{IalphaXest} and Lemma \ref{Fourier} we deduce
\begin{equation}\label{Phiest}
\Phi\ll\frac{\varepsilon}{\Delta^2}\,.
\end{equation}
Bearing in mind \eqref{Delta}, \eqref{Gamma1-JX}, \eqref{JXest1},
\eqref{Best} and  \eqref{Phiest} we obtain
\begin{equation}\label{Gamma1est}
\Gamma_1(X)\gg\varepsilon X^2\,.
\end{equation}

\section{Upper bound of $\mathbf{\Gamma_2(X)}$}\label{SectionGamma2}
\indent

Suppose that
\begin{equation*}
\bigg|\alpha -\frac{a}{q}\bigg|\leq\frac{1}{q^2}\,,\quad (a, q)=1
\end{equation*}
with
\begin{equation}\label{Intq}
q\in\left[X^{\frac{1}{13}}, X^{\frac{12}{13}}\right]\,.
\end{equation}
Then \eqref{SigmaalphaX}, \eqref{Intq} and  Lemma \ref{Expsumest} yield
\begin{equation}\label{Sigmaest}
\Sigma(\alpha,\,X)\ll X^{\frac{25}{26}}\log^4X\,.
\end{equation}
Now \eqref{SOmegaSigma}, \eqref{Sigmaest} and Lemma \ref{Omegaest} give us
\begin{equation}\label{Salphaest}
S(\alpha,\,X)\ll X^{\frac{37-12\gamma}{26}}\log^5X\,.
\end{equation}
Let
\begin{equation}\label{mathfrakS}
\mathfrak{S}(t,X)=\min\left\{\left|S(\lambda_{1}t,\,X)\right|,\left|S(\lambda_2 t,\,X)\right|\right\}\,.
\end{equation}
We shall prove the following lemma.
\begin{lemma}\label{mathfrakSest} Let $t,\,X,\,\lambda_1,\,\lambda_2\in \mathbb{R}$,
\begin{equation}\label{tdeltaH}
\Delta\leq|t|\leq H\,,
\end{equation}
where $\Delta$ and $H$ are denoted by
\eqref{Delta} and \eqref{H}, $\lambda_1/\lambda_2 \in \mathbb{R}\backslash \mathbb{Q}$
 and $\mathfrak{S}(t,X)$ is defined by \eqref{mathfrakS}.
Then there exists a sequence of real numbers
$X_1,\,X_2,\ldots \to \infty $ such that
\begin{equation*}
\mathfrak{S}(t, X_j)\ll X_j^{\frac{37-12\gamma}{26}}\log^5X_j\,,\quad j=1,2,\dots\,.
\end{equation*}
\end{lemma}
\begin{proof} Our aim is to prove that there exists a sequence $X_1,\,X_2,\,...\to \infty $ such that
for each $j=1,2,\ldots $ at least one of the numbers $\lambda_{1}t$ and $\lambda_{2}t$ with t,
subject to \eqref{tdeltaH} can be approximated by rational numbers with denominators,
satisfying \eqref{Intq}. Then the proof follows from \eqref{Salphaest} and \eqref{mathfrakS}.

Since $\lambda_1,\lambda_2,\lambda_3$ are not all of the same sign one can assume that
$\lambda_1>0,\,\lambda_2>0$ and $\lambda_3<0$.
Let us notice that there exist $a_1,\,q_1\in \mathbb{Z}$,
such that
\begin{equation}\label{lambda1a1q1}
  \bigg|\lambda_1t-\frac{a_1}{q_1}\bigg|<\frac{1}{q_1q_0^2}\,,
  \quad\quad (a_1,\,q_1)=1,\quad\quad 1\leq q_1\leq q_0^2,\quad\quad a_1\ne 0\,.
\end{equation}
From Dirichlet's approximation theorem it follows the existence
of integers $a_1$ and $q_1$, satisfying the first three conditions.
If $a_1=0$ then
\begin{equation*}
|\lambda_1t|< \frac{1}{q_1q_0^2}
\end{equation*}
and \eqref{tdeltaH} gives us
\begin{equation*}
\lambda_1\Delta< \lambda_1|t|< \frac{1}{q_0^2}\,,\quad\quad
q_0^2< \frac{1}{\lambda_1\Delta}\,.
\end{equation*}
The last inequality,  (\ref{X}) and (\ref{Delta}) yield
\begin{equation*}
X^{\frac{12}{13}}<\frac{X^{\frac{12}{13}}}{\lambda_1\log X}\,,
\end{equation*}
which is impossible for large $X$.
Therefore $a_1\ne 0$.
By analogy there exist $a_2,\,q_2\in \mathbb{Z}$,
such that
\begin{equation}\label{lambda2a2q2}
  \bigg|\lambda_2t-\frac{a_2}{q_2}\bigg|<\frac{1}{q_2q_0^2}\,,
  \quad\quad (a_2,\,q_2)=1,\quad\quad 1\leq q_2\leq q_0^2,\quad\quad a_2\ne 0\,.
\end{equation}
If $q_i\in\left[X^{\frac{1}{13}},\,X^{\frac{12}{13}}\right]$
for $i=1$ or $i=2$, then the proof is completed.
By \eqref{X}, \eqref{lambda1a1q1} and \eqref{lambda2a2q2} we deduce
\begin{equation*}
q_i\le X^{\frac{12}{13}}=q_0^2\,,\quad i=1,2\,.
\end{equation*}
It remains to show that the case $q_i<X^{\frac{1}{13}}\,,i=1,2$ is impossible.
Assume that
\begin{equation}\label{impossible}
q_i<X^{\frac{1}{13}}\,,\quad i=1,2\,.
\end{equation}
From \eqref{varepsilon}, \eqref{H}, \eqref{tdeltaH} -- \eqref{impossible}
it follows
\begin{align}
  & 1\le |a_i|<\frac{1}{q_0^2}+q_i\lambda_i|t|< \frac{1}{q_0^2}+q_i\lambda_i H\,,\nonumber\\
\label{ai}
& 1\le |a_i|<\frac{1}{q_0^2}+\lambda_iX^{\frac{38\gamma-35}{26}}(\log X)^{-8}\,,\quad i=1,\,2\,.
\end{align}
We have
\begin{equation}\label{lambda12}
  \frac{\lambda_1}{\lambda_2}=\frac{\lambda_1t}{\lambda_2t}=
  \frac{\frac{a_1}{q_1}+\bigg(\lambda_1t-\frac{a_1}{q_1}\bigg)}{\frac{a_2}{q_2}+\bigg(\lambda_2t-\frac{a_2}{q_2}\bigg)}=
  \frac{a_1q_2}{a_2q_1}\cdot\frac{1+\mathfrak{X}_1}{1+\mathfrak{X}_2}\,,
\end{equation}
where
\begin{equation}\label{mathfrakX}
\mathfrak{X}_i=\dfrac{q_i}{a_i}\bigg(\lambda_it-\dfrac{a_i}{q_i}\bigg)\,,\; i=1,\,2.
\end{equation}
Bearing in mind  \eqref{lambda1a1q1}, \eqref{lambda2a2q2}, \eqref{lambda12} and \eqref{mathfrakX} we get
\begin{align}
  &|\mathfrak{X}_i|< \frac{q_i}{|a_i|}\cdot \frac{1}{q_iq_0^2}=\frac{1}{|a_i|q_0^2}\le \frac{1}{q_0^2}\,,\quad i=1,2\,,\nonumber\\
\label{lambd12}
  &\frac{\lambda_1}{\lambda_2}=\frac{a_1q_2}{a_2q_1}\cdot
  \frac{ 1+\mathcal{O}\bigg(\frac{1}{q_0^2}\bigg)}{ 1+\mathcal{O}\bigg(\frac{1}{q_0^2}\bigg)}=
 \frac{a_1q_2}{a_2q_1}\bigg(1+\mathcal{O}\bigg(\frac{1}{q_0^2}\bigg)\bigg)\,.\notag
\end{align}
Thus
\begin{equation*}
\frac{a_1q_2}{a_2q_1}=\mathcal{O}(1)
\end{equation*}
and
\begin{equation}\label{lambd12new}
\frac{\lambda_1}{\lambda_2}=\frac{a_1q_2}{a_2q_1}+\mathcal{O}\bigg(\frac{1}{q_0^2}\bigg)\,.
\end{equation}
Therefore, both fractions $\displaystyle \frac{a_0}{q_0}$ and $\displaystyle  \frac{a_1q_2}{a_2q_1}$
approximate  $\displaystyle  \frac{\lambda_1}{\lambda_2}$.
Using \eqref{X}, \eqref{lambda1a1q1}, \eqref{impossible} and inequality \eqref{ai} with $i=2$ we obtain
\begin{equation}\label{a2q1}
  |a_2|q_1<1+\lambda_2X^{\frac{38\gamma-33}{26}}(\log X)^{-8}<\frac{q_0}{\log X}\,.
\end{equation}
Consequently $|a_2|q_1\ne q_0$  and  $\displaystyle  \frac{a_0}{q_0}\neq\frac{a_1q_2}{a_2q_1}$.
Now \eqref{a2q1} implies
\begin{equation}\label{contradiction}
  \bigg|\frac{a_0}{q_0}-\frac{a_1q_2}{a_2q_1}\bigg|=
  \frac{|a_0 a_2q_1-a_1q_2q_0|}{|a_2|q_1q_0}\ge \frac{1}{|a_2|q_1q_0}>\frac{\log X}{q_0^2}\,.
\end{equation}
On the other hand, from \eqref{lambda12a0q0} and \eqref{lambd12new} we deduce
\begin{equation*}
  \bigg|\frac{a_0}{q_0}-\frac{a_1q_2}{a_2q_1}\bigg|\le \bigg|\frac{a_0}{q_0}-\frac{\lambda_1}{\lambda_2}\bigg|+
  \bigg|\frac{\lambda_1}{\lambda_2}-\frac{a_1q_2}{a_2q_1}\bigg|\ll \frac{1}{q_0^2}\,,
\end{equation*}
which contradicts \eqref{contradiction}.
This rejects the assumption \eqref{impossible}.
Let $q_0^{(1)},\,q_0^{(2)},\,\ldots$ be an infinite sequence of values of $q_0$, satisfying \eqref{lambda12a0q0}.
Then using \eqref{X} one gets an infinite sequence $X_1,\,X_2,\,\ldots $ of values of $X$, such that
at least one of the numbers $\lambda_{1}t$ and $\lambda_{2}t$ can be approximated by
rational numbers with denominators, satisfying \eqref{Intq}.
Hence, the proof is completed.
\end{proof}

Taking into account \eqref{Gamma2}, \eqref{mathfrakS}, Lemma \ref{Fourier}
and Lemma \ref{mathfrakSest}  we deduce
\begin{align}\label{Gamma2est1}
\Gamma_2(X_j)&\ll\varepsilon\int\limits_{\Delta\leq|t|\leq H}\mathfrak{S}(t, X_j)
\Big(\big|S(\lambda_1 t, X_j)S(\lambda_3 t, X_j)\big|
+\big|S(\lambda_2 t, X_j)S(\lambda_3 t, X_j)\big|\Big)\,dt\nonumber\\
&\ll\varepsilon\int\limits_{\Delta\leq|t|\leq H}\mathfrak{S}(t, X_j)
\Big(\big|S(\lambda_1 t, X_j)\big|^2
+\big|S(\lambda_2 t, X_j)\big|^2+\big|S(\lambda_3 t, X_j)\big|^2\Big)\,dt\nonumber\\
&\ll\varepsilon X_j^{\frac{37-12\gamma}{26}}(\log X_j)^5T_k\,,
\end{align}
where
\begin{equation*}
T_k=\int\limits_{\Delta}^H\big|S(\lambda_k t, X_j)\big|^2\,dt.
\end{equation*}
Using Lemma \ref{intSintI} (iii) and working as in (\cite{DimTod}, p. 17 -- 18) we obtain
\begin{equation}\label{Tkest}
T_k\ll HX_j^{2-\gamma}\log^2X_j\,.
\end{equation}
From \eqref{varepsilon}, \eqref{H}, \eqref{Gamma2est1}, \eqref{Tkest} we get
\begin{equation}\label{Gamma2est2}
\Gamma_2(X_j)\ll X_j^{\frac{37-12\gamma}{26}}X_j^{2-\gamma}\log^9X_j
\ll  X_j^{\frac{89-38\gamma}{26}}\log^9X_j\ll\frac{\varepsilon X_j^2}{\log X_j}\,.
\end{equation}

\section{Upper bound of $\mathbf{\Gamma_3(X)}$}\label{SectionGamma3}
\indent

By \eqref{SalphaX}, \eqref{Gamma3} and Lemma \ref{Fourier} it follows
\begin{equation}\label{Gamma3est1}
\Gamma_3(X)\ll X^{3-3\gamma}\int\limits_{H}^{\infty}\frac{1}{t}\bigg(\frac{k}{2\pi t\varepsilon/8}\bigg)^k \,dt
=\frac{X^{3-3\gamma}}{k}\bigg(\frac{4k}{\pi\varepsilon H}\bigg)^k\,.
\end{equation}
Choosing $k=[\log X]$ from \eqref{H} and \eqref{Gamma3est1} we obtain
\begin{equation}\label{Gamma3est}
\Gamma_3(X)\ll1\,.
\end{equation}

\section{Proof of the Theorem}\label{Sectionfinal}
\indent

Summarizing  \eqref{varepsilon}, \eqref{Gammadecomp}, \eqref{Gamma1est},
\eqref{Gamma2est2} and \eqref{Gamma3est} we deduce
\begin{equation*}
\Gamma(X_j)\gg\varepsilon X_j^2=X_j^{\frac{89-38\gamma}{26}}\log^{10}X_j\,.
\end{equation*}
The last estimation implies
\begin{equation}\label{Lowerbound}
\Gamma(X_j) \rightarrow\infty \quad \mbox{ as } \quad X_j\rightarrow\infty\,.
\end{equation}
Bearing in mind  \eqref{Gamma} and \eqref{Lowerbound} we establish Theorem \ref{Theorem}.

\vskip20pt
\footnotesize
\begin{flushleft}
S. I. Dimitrov\\
Faculty of Applied Mathematics and Informatics\\
Technical University of Sofia \\
8, St.Kliment Ohridski Blvd. \\
1756 Sofia, BULGARIA\\
e-mail: sdimitrov@tu-sofia.bg\\
\end{flushleft}

\end{document}